\documentclass[11pt]{article}
\usepackage{amsmath,amssymb,amsthm}
\DeclareMathOperator{\End}{End}
\DeclareMathOperator{\Hom}{Hom}

\DeclareMathOperator{\rad}{rad}
\usepackage{geometry}
\usepackage{graphicx}
\usepackage{float}
\usepackage{caption}
\usepackage[hidelinks]{hyperref}
\usepackage[T1]{fontenc}
\usepackage[utf8]{inputenc}
\theoremstyle{remark}
\usepackage{lineno}

\usepackage{array}
\usepackage{tabularx}
\theoremstyle{plain}
\newtheorem{theorem}{Theorem}
\newtheorem{proposition}[theorem]{Proposition}
\newtheorem{lemma}[theorem]{Lemma}

\newtheorem{corollary}[theorem]{Corollary}
\usepackage{listings}
\usepackage{xcolor}
\DeclareMathOperator{\rank}{rank}
\usepackage{booktabs}
\usepackage{longtable}
\usepackage{array}
\usepackage{threeparttablex}
\usepackage{pdflscape}
\usepackage{tikz}
\usetikzlibrary{arrows.meta,decorations.pathreplacing,positioning}

\theoremstyle{definition}
\newtheorem{definition}[theorem]{Definition}

\newtheoremstyle{myremark}
{6pt}
{6pt}
{\normalfont}
{}
{\bfseries}
{.}
{0.5em}
{}

\theoremstyle{myremark}
\newtheorem{remark}{Remark}
\begin{document}

\title{Stack Constructions and Equality Cases for Laffey's Bound\\
	for Maximal Commutative Matrix Algebras}

\author{Ma\l gorzata Nowak-K\k{e}pczyk\footnote{Corresponding Author}\\
	$\rule{0pt}{13pt}$Institute of Informatics\\
	John Paul II Catholic University of Lublin, Poland\\
	email: \texttt{malnow@kul.pl}}

\date{}
\maketitle

\begin{abstract}
	We study equality cases in Laffey's lower bound for
	local maximal commutative subalgebras of matrix algebras
	with cube-zero radical.
	
	We introduce a stacking operation that combines
	algebras with common outer Loewy layers while preserving
	maximal commutativity. This yields standard equality
	bricks $B_{a,c}=(a,ac+1,c;a+c)$ for all $a,c\ge2$
	and reduces further constructions to smaller building
	blocks.
	
	We establish the existence of all minimal equality
	bricks on the balanced branch and obtain sharp
	realizations throughout the $(4,5)$-sector.
	We also construct the first nontrivial nearly balanced
	root over arbitrary fields.
	
	We determine the balanced--nearly-balanced envelope
	of the numerical Laffey profile and give an arithmetic
	criterion for preservation of sharpness under stacking.
	As a consequence, the profile is attained in every
	matrix dimension $14\le n\le52$ and $93\le n\le114$,
	as well as along five infinite families over arbitrary
	fields.
	
	The main existence and interval results are proved
	in characteristic zero; the field-independent
	constructions are stated explicitly.
\end{abstract}



\medskip
\noindent
Throughout this paper, $K$ denotes a field of characteristic zero unless explicitly stated otherwise. Results valid over arbitrary fields will be indicated accordingly.

\section{Introduction}
\label{sec:introduction}

The problem of determining the smallest possible dimension
of a maximal commutative subalgebra of $M_n(K)$ has a long
history. 
The diagonal algebra provides a natural benchmark
of dimension $n$, suggesting that every maximal
commutative subalgebra of $M_n(K)$ might have
dimension at least $n$. However, Courter constructed a $13$-dimensional
maximal commutative subalgebra of $M_{14}(K)$
with $J^3=0$ and $J^2\ne0$ \cite{Courter1965,Courter1973},
disproving this expectation.

Gustafson subsequently established a general lower bound
of order $n^{2/3}$ \cite{Gustafson1976}, which Laffey
improved to
\[
\dim A>(2n)^{2/3}-1
\]
for every maximal commutative subalgebra of $M_n(K)$
\cite{Laffey1985}. For algebras with cube-zero radical,
Laffey obtained the stronger estimate
\[
\dim A\ge\lfloor 3n^{2/3}-4\rfloor,
\]
which is best possible for infinitely many matrix
dimensions.

In the local cube-zero case, the action on $K^n$
determines a three-layer Loewy filtration with signature
$(a,b,c)$. Laffey's underlying signature-dependent
bound takes the form
\begin{equation}
	\label{eq:intro-laffey}
	\dim A\ge
	1+ac+
	\left\lceil\frac{b(a+c)}{ac+1}\right\rceil.
\end{equation}
This estimate arises from the Loewy block form and a
dimension count. Maximal commutativity forces the full
bottom-left block into the algebra, while the remaining
commutation relations form a linear system whose solution
space is constrained by the centralizer.

Laffey also constructed an infinite family attaining
the signature-dependent bound, with matrix dimensions
$n=k^3+3k$ and algebra dimensions $3k^2+1$.
These examples demonstrate sharpness along an infinite
sequence, but do not settle the realizability of equality
at other Loewy signatures or matrix dimensions.

Subsequent research explored both constructions and
structural properties of maximal commutative matrix
algebras. Brown and Call investigated their structure
and construction \cite{BrownCall1993}. Song showed that
Courter's numerical parameters do not determine a unique
algebra: there exist at least two non-isomorphic
$13$-dimensional maximal commutative subalgebras of
$M_{14}(K)$ with radical of nilpotency index three
\cite{Song1997}. In a complementary direction,
Ko\v{s}ir studied algebras generated by commuting
nilpotent matrices with a one-dimensional common kernel
\cite{Kosir1997}.

These developments distinguish two questions:
which dimensions can be attained, and what structures
can realize them? The present paper focuses on the
attainment of Laffey's bound, using structural information
to construct and combine equality cases.


\medskip

Earlier module-theoretic constructions of
self-centralizing commutative matrix algebras
include the extensions $C_1$ and $C_2$
of Nowak-Morales\footnote{Nowak-Morales
	is the former surname of the present author.}
\cite{NowakMorales1999}.
The former uses trivial extensions by
repeated modules, while the latter adjoins
a generator subject to a quadratic relation.

Our approach is based on a reduced-centralizer description
of local maximal commutative algebras with cube-zero radical.
We introduce a stacking operation that combines compatible
algebras by identifying their outer Loewy layers and the
full bottom-left block while adding their middle layers.
The operation preserves maximal commutativity and yields,
for every $a,c\ge2$, a standard equality brick
\[
B_{a,c}=(a,ac+1,c;a+c).
\]
Its iterated stacks recover the numerical parameters of
Laffey's classical family and provide further equality
cases. We also study smaller primitive equality bricks.
Using an explicit block-matrix induction, we construct
all minimal equality bricks on the balanced branch in
characteristic zero, including
$E_{k,k}=(k,(k^2+1)/2,k;k)$ for odd $k$.
For the nearly balanced pair $(4,5)$, we construct
$E_{4,5}=(4,7,5;3)$ and obtain sharp realizations
for every middle dimension $b\ge4$ over arbitrary fields.

To study attainment at a prescribed matrix dimension,
we introduce the numerical profile
\[
L_3(n)=
\min_{\substack{a+b+c=n\\a,b,c\ge1}}
\left\{
1+ac+
\left\lceil\frac{b(a+c)}{ac+1}\right\rceil
\right\}.
\]
We determine its balanced--nearly-balanced outer-pair
envelope and establish an arithmetic criterion for
preserving sharpness under stacking. Combining these
results, we prove that $L_3(n)$ is attained for every
$14\le n\le52$ and $93\le n\le114$ over a field of characteristic zero, and along five infinite families of matrix dimensions over arbitrary
fields. The exceptional finite constructions are
supported by exact reduced-centralizer certificates.
The realizability of further primitive equality bricks
and the eventual attainment of $L_3(n)$ remain open.

\section{Loewy form and Laffey's signature bound}
\label{sec:loewy-laffey}

Throughout this section, let
\(A\subseteq M_n(K)\) be a local maximal commutative algebra
over an algebraically closed field \(K\), and put
\(J=\rad(A)\).
Since \(A/J\simeq K\), we have
\[
A=KI_n\oplus J
\]
as a vector space.

\subsection{Loewy form}

\begin{definition}[Loewy signature]
	Suppose that \(J\) has nilpotency degree \(r\), that is,
	\[
	J^rK^n=0,
	\qquad
	J^{r-1}K^n\neq0.
	\]
	The radical filtration
	\[
	K^n\supset JK^n\supset J^2K^n
	\supset\cdots\supset J^{r-1}K^n\supset0
	\]
	determines integers
	\[
	n_i=\dim_K(J^{i-1}K^n/J^iK^n),
	\qquad 1\le i\le r.
	\]
	We call
	\[
	\operatorname{sig}_L(A)=(n_1,\ldots,n_r)
	\]
	the \emph{Loewy signature} of the representation
	\(A\subseteq M_n(K)\).
\end{definition}

Choose vector-space complements
\[
J^{i-1}K^n=V_i\oplus J^iK^n,
\qquad 1\le i\le r.
\]
Then
\[
K^n=V_1\oplus\cdots\oplus V_r,
\qquad \dim V_i=n_i,
\]
and
\[
JV_i\subseteq J^iK^n
=V_{i+1}\oplus\cdots\oplus V_r.
\]
Thus every element of \(J\) is strictly lower block
triangular with respect to an adapted basis, with
diagonal block sizes \(n_1,\ldots,n_r\).
We call this a \emph{Loewy form} of \(A\).

The extreme block is forced by maximal commutativity.

\begin{proposition}[Full bottom-left block]
	\label{prop:full-bottom-left-block}
	Let \(A\subseteq M_n(K)\) be local and maximal
	commutative, with Loewy decomposition
	\[
	K^n=V_1\oplus\cdots\oplus V_r,
	\qquad r\ge2.
	\]
	Then the space \(\Sigma\) of matrices supported
	in the bottom-left block,
	\[
	\Sigma=\Hom(V_1,V_r),
	\]
	is contained in \(J\). Consequently, for any
	vector-space complement \(\Lambda\) of \(\Sigma\)
	in \(J\),
	\[
	A=KI_n\oplus\Lambda\oplus\Sigma,
	\qquad
	\dim A=1+\dim\Lambda+n_1n_r.
	\]
\end{proposition}

\begin{proof}
	Let \(S\) be supported in the bottom-left block.
	Every \(R\in J\) maps \(K^n\) into \(JK^n\), on
	which \(S\) vanishes. Hence \(SR=0\).
	Conversely, \(S(K^n)\subseteq V_r
	\subseteq J^{r-1}K^n\), so \(RS=0\).
	Thus \(S\) commutes with \(A\).
	
	Maximal commutativity gives
	\[
	S\in C_{M_n(K)}(A)=A.
	\]
	Since \(S\) is nilpotent, its image in the field
	\(A/J\simeq K\) vanishes. Therefore \(S\in J\).
	
	The dimension formula follows by taking a
	vector-space complement of \(\Sigma\) in \(J\).
\end{proof}

\subsection{The cube-zero case}

We now specialize to
\[
J^3=0,
\qquad J^2K^n\neq0,
\]
and write
\[
\operatorname{sig}_L(A)=(a,b,c).
\]
Choose a Loewy decomposition
\[
K^n=U\oplus W\oplus Z,
\qquad
\dim U=a,\quad \dim W=b,\quad \dim Z=c.
\]

By Proposition~\ref{prop:full-bottom-left-block},
\[
\Sigma=\Hom(U,Z)\subseteq J.
\]
We may choose a complement \(\Lambda\) of
\(\Sigma\) in \(J\) whose elements have
vanishing bottom-left block. Each
\(\lambda\in\Lambda\) then has the unique form
\[
\lambda(A,B)=
\begin{pmatrix}
	0&0&0\\
	A&0&0\\
	0&B&0
\end{pmatrix},
\]
where
\[
A\in\Hom(U,W),
\qquad B\in\Hom(W,Z).
\]

Put \(d=\dim\Lambda\). We record the numerical
data by the \emph{parameter type}
\[
(a,b,c;d).
\]
In particular,
\begin{equation}
	\label{eq:type-dimension}
	\dim A=1+ac+d.
\end{equation}

A parameter type is called \emph{realizable}
if it occurs for a local maximal commutative
algebra with cube-zero radical and Loewy
signature \((a,b,c)\).

Choose a basis
\[
\lambda_i=\lambda(A_i,B_i),
\qquad 1\le i\le d,
\]
of \(\Lambda\). Since
\[
\lambda_i\lambda_j=
\begin{pmatrix}
	0&0&0\\
	0&0&0\\
	B_iA_j&0&0
\end{pmatrix},
\]
commutativity is equivalent to
\begin{equation}
	\label{eq:commutativity-pairs}
	B_iA_j=B_jA_i
	\qquad(1\le i,j\le d).
\end{equation}

The Loewy filtration imposes two spanning
conditions. First, \(JK^n=W\oplus Z\) gives
\begin{equation}
	\label{eq:A-images-span}
	W=\sum_{i=1}^d\operatorname{im}A_i.
\end{equation}
Second, \(J^2K^n=Z\) gives
\begin{equation}
	\label{eq:product-span}
	Z=\sum_{i,j=1}^d
	\operatorname{im}(B_iA_j).
\end{equation}

\subsection{The reduced centralizer}

Maximal commutativity can be tested entirely
in terms of the generator pairs.

\begin{definition}[Reduced centralizer]
	\label{def:reduced-centralizer}
	For a cube-zero system \((A_i,B_i)\),
	\(1\le i\le d\), define
	\[
	C_{\mathrm{red}}(A)=
	\left\{
	(P,S)\in
	\Hom(U,W)\oplus\Hom(W,Z):
	SA_i=B_iP\ \text{for all }i
	\right\}.
	\]
\end{definition}

Equivalently, consider the linear map
\begin{equation}
	\label{eq:reduced-centralizer-map}
	\Psi_A:
	\Hom(U,W)\oplus\Hom(W,Z)
	\longrightarrow\Hom(U,Z)^d
\end{equation}
given by
\[
\Psi_A(P,S)=(SA_i-B_iP)_{i=1}^d.
\]
Then \(C_{\mathrm{red}}(A)=\ker\Psi_A\).

\begin{proposition}[Reduced centralizer criterion]
	\label{prop:reduced-centralizer}
	Let
	\[
	A=KI_n\oplus\Lambda\oplus\Sigma
	\]
	be a commutative cube-zero algebra in the
	above block form, where
	\(\Sigma=\Hom(U,Z)\), the \(d\) generator
	pairs form a basis of \(\Lambda\), and
	\[
	W=\sum_{i=1}^d\operatorname{im}A_i.
	\]
	Then
	\[
	\dim C_{M_n(K)}(A)
	=1+ac+\dim C_{\mathrm{red}}(A).
	\]
	In particular,
	\[
	A\text{ is maximal commutative}
	\quad\Longleftrightarrow\quad
	\dim C_{\mathrm{red}}(A)=d.
	\]
\end{proposition}

\begin{proof}
	Let \(X\in C_{M_n(K)}(A)\).
	Commutation with the full block
	\(\Sigma=\Hom(U,Z)\) implies that, after
	subtracting a scalar matrix,
	\[
	X=
	\begin{pmatrix}
		0&0&0\\
		P&Q&0\\
		R&S&0
	\end{pmatrix},
	\]
	where
	\[
	P\in\Hom(U,W),\quad Q\in\End(W),
	\]
	\[
	R\in\Hom(U,Z),\quad S\in\Hom(W,Z).
	\]
	
	Commutation with each \(\lambda_i\) gives
	\[
	QA_i=0,\qquad B_iQ=0,\qquad
	SA_i=B_iP.
	\]
	Since the images of the \(A_i\) span \(W\),
	we obtain \(Q=0\).
	
	The block \(R\in\Sigma\) is arbitrary, while
	the remaining equations are precisely
	\((P,S)\in C_{\mathrm{red}}(A)\).
	Consequently,
	\[
	C_{M_n(K)}(A)
	\simeq KI_n\oplus\Sigma
	\oplus C_{\mathrm{red}}(A)
	\]
	as vector spaces, and
	\[
	\dim C_{M_n(K)}(A)
	=1+ac+\dim C_{\mathrm{red}}(A).
	\]
	
	By \eqref{eq:commutativity-pairs},
	all \(d\) generator pairs belong to
	\(C_{\mathrm{red}}(A)\), and they are
	linearly independent. Thus
	\[
	\dim C_{\mathrm{red}}(A)\ge d.
	\]
	Together with \eqref{eq:type-dimension},
	this proves the criterion.
\end{proof}

\subsection{Laffey's signature bound}

We recover the signature-dependent lower
bound of Laffey \cite{Laffey1985} directly
from the reduced-centralizer criterion.

\begin{proposition}[Laffey's bound]
	\label{prop:laffey-signature}
	Let \(A\subseteq M_n(K)\) be a local
	maximal commutative algebra with \(J^3=0\)
	and Loewy signature \((a,b,c)\).
	Then
	\begin{equation}
		\label{eq:Laffey_signature}
		\dim A\ge
		1+ac+
		\left\lceil
		\frac{b(a+c)}{ac+1}
		\right\rceil.
	\end{equation}
\end{proposition}

\begin{proof}
	Write the parameter type as \((a,b,c;d)\).
	By Proposition~\ref{prop:reduced-centralizer},
	\[
	\dim\ker\Psi_A=d.
	\]
	The domain of \(\Psi_A\) has dimension
	\(b(a+c)\), while its codomain has
	dimension \(dac\). Hence rank--nullity gives
	\[
	b(a+c)-d
	=\rank\Psi_A\le dac.
	\]
	Therefore
	\[
	d\ge
	\left\lceil\frac{b(a+c)}{ac+1}\right\rceil.
	\]
	The assertion follows from
	\(\dim A=1+ac+d\).
\end{proof}

For fixed \(n\), define the
\emph{numerical Laffey profile}
\begin{equation}
	\label{eq:L3-profile}
	L_3(n)=
	\min_{\substack{a+b+c=n\\a,b,c\ge1}}
	\left\{
	1+ac+
	\left\lceil
	\frac{b(a+c)}{ac+1}
	\right\rceil
	\right\}.
\end{equation}
This is a minimum over numerical signatures;
it does not require their realizability.
Every local maximal commutative algebra
\(A\subseteq M_n(K)\) with cube-zero radical
satisfies
\[
\dim A\ge L_3(n).
\]

If \(a=1\), then
\[
1+ac+
\left\lceil
\frac{b(a+c)}{ac+1}
\right\rceil
=1+c+b=n,
\]
and likewise when \(c=1\).
Thus a signature yielding a value below
\(n\) must have \(a,c\ge2\).

Laffey's analysis gives
\[
L_3(n)\ge n\qquad(n\le13),
\]
whereas
\[
L_3(14)=13.
\]
The latter value is attained by Courter's
classical algebra, with Loewy signature
\((2,10,2)\).

The problem in the remainder of the paper
is therefore not to improve
\eqref{eq:Laffey_signature}, but to
determine when this bound is sharp and
to construct maximal commutative algebras
realizing its equality cases.

\section{Stack algebras}
\label{sec:stack}

We now introduce a construction that combines local
cube-zero algebras by identifying their outer Loewy
layers while retaining separate middle layers.
This is closely related in spirit to classical
block constructions, including those of Courter
\cite{Courter1973}. Its formulation in terms of
reduced centralizers will be particularly useful
for constructing and combining equality cases.

Let
\[
\mathcal A_i\subseteq M_{a+b_i+c}(K),
\qquad i=1,\ldots,q,
\]
be local maximal commutative algebras with cube-zero
radicals and parameter types
\[
(a,b_i,c;d_i).
\]
Choose Loewy decompositions
\[
K^{a+b_i+c}=U\oplus W_i\oplus Z,
\]
where
\[
\dim U=a,\qquad
\dim W_i=b_i,\qquad
\dim Z=c.
\]

By Proposition~\ref{prop:full-bottom-left-block},
each algebra contains the full bottom-left block
\[
\Sigma=\Hom(U,Z).
\]
Choosing complements of \(\Sigma\) in the radicals,
we may write
\[
\mathcal A_i=KI\oplus\Lambda_i\oplus\Sigma,
\qquad
\dim\Lambda_i=d_i,
\]
where the elements of \(\Lambda_i\) are represented
by pairs
\[
(A_{i,\alpha},B_{i,\alpha}),
\qquad
A_{i,\alpha}:U\to W_i,\quad
B_{i,\alpha}:W_i\to Z.
\]

\begin{definition}[Compatible algebras]
	\label{def:compatible-algebras}
	The algebras \(\mathcal A_1,\ldots,\mathcal A_q\)
	are called \emph{compatible} if their Loewy
	realizations use the same outer spaces \(U\) and
	\(Z\), and hence the same block
	\(\Sigma=\Hom(U,Z)\).
\end{definition}

Since compatibility requires only vector-space
identifications of the outer Loewy layers,
algebras with the same outer dimensions may
always be placed in compatible form.

\begin{definition}[Stack algebra]
	\label{def:stack-algebra}
	The \emph{stack}
	\[
	\mathcal A_1\star\cdots\star\mathcal A_q
	\]
	of compatible local cube-zero algebras is the
	algebra acting on
	\[
	U\oplus W_1\oplus\cdots\oplus W_q\oplus Z
	\]
	whose radical consists of the common block
	\(\Sigma=\Hom(U,Z)\) and the spaces \(\Lambda_i\),
	each acting through its own middle summand \(W_i\).
	
	For two factors, its radical elements have the
	form
	\[
	\begin{pmatrix}
		0&0&0&0\\
		A^{(1)}&0&0&0\\
		A^{(2)}&0&0&0\\
		R&B^{(1)}&B^{(2)}&0
	\end{pmatrix},
	\]
	where
	\[
	(A^{(i)},B^{(i)})\in\Lambda_i,
	\qquad R\in\Sigma.
	\]
\end{definition}

Thus stacking differs from an ordinary direct
sum: the scalar identity, the outer Loewy layers,
and the full bottom-left block are shared, while
the middle layers are added.

\begin{theorem}[Stack theorem]
	\label{thm:stack}
	Let \(\mathcal A_1,\ldots,\mathcal A_q\) be
	compatible local maximal commutative cube-zero
	algebras of parameter types
	\[
	(a,b_i,c;d_i),
	\qquad i=1,\ldots,q.
	\]
	Then their stack
	\[
	\mathcal A
	=\mathcal A_1\star\cdots\star\mathcal A_q
	\]
	is again local and maximal commutative, with
	cube-zero radical and parameter type
	\[
	\left(
	a,\sum_{i=1}^q b_i,c;
	\sum_{i=1}^q d_i
	\right).
	\]
	In particular,
	\begin{equation}
		\label{eq:stack-dimension}
		\dim\mathcal A
		=1+ac+\sum_{i=1}^q d_i,
	\end{equation}
	or equivalently,
	\begin{equation}
		\label{eq:stack-dimension-factors}
		\dim\mathcal A
		=\sum_{i=1}^q\dim\mathcal A_i
		-(q-1)(1+ac).
	\end{equation}
\end{theorem}

\begin{proof}
	Put
	\[
	W=W_1\oplus\cdots\oplus W_q.
	\]
	The radical of the stack is
	\[
	J=\Lambda_1\oplus\cdots
	\oplus\Lambda_q\oplus\Sigma.
	\]
	
	Generators supported on different middle
	summands annihilate one another, while
	generators from the same summand commute
	by assumption. Moreover, the common block
	\(\Sigma\) annihilates the radical on both
	sides. Hence the stack is commutative.
	
	For each \(i\), the images of the maps
	\(A_{i,\alpha}:U\to W_i\) span \(W_i\).
	Consequently,
	\[
	JK^n=W\oplus Z.
	\]
	Since each constituent has Loewy signature
	\((a,b_i,c)\), its products generate \(Z\).
	These products are unchanged in the stack,
	whereas mixed products vanish. Thus
	\[
	J^2K^n=Z,\qquad J^3K^n=0,
	\]
	and the Loewy signature is
	\[
	\left(a,\sum_{i=1}^q b_i,c\right).
	\]
	The algebra is local because it is of the
	form \(KI\oplus J\) with \(J\) nilpotent.
	
	It remains to prove maximal commutativity.
	Write a pair in the reduced centralizer
	of the stack as
	\[
	P=
	\begin{pmatrix}
		P_1\\ \vdots\\ P_q
	\end{pmatrix}:U\to W,
	\qquad
	S=
	\begin{pmatrix}
		S_1&\cdots&S_q
	\end{pmatrix}:W\to Z,
	\]
	where
	\[
	P_i:U\to W_i,\qquad S_i:W_i\to Z.
	\]
	
	For a generator
	\((A_{i,\alpha},B_{i,\alpha})\)
	supported on \(W_i\), the equation
	\[
	SA_{i,\alpha}=B_{i,\alpha}P
	\]
	reduces to
	\[
	S_iA_{i,\alpha}
	=B_{i,\alpha}P_i.
	\]
	Therefore the reduced-centralizer equations
	separate completely:
	\[
	C_{\mathrm{red}}(\mathcal A)
	=
	\bigoplus_{i=1}^q
	C_{\mathrm{red}}(\mathcal A_i).
	\]
	
	Each constituent is maximal commutative.
	By Proposition~\ref{prop:reduced-centralizer},
	\[
	\dim C_{\mathrm{red}}(\mathcal A_i)=d_i.
	\]
	It follows that
	\[
	\dim C_{\mathrm{red}}(\mathcal A)
	=\sum_{i=1}^q d_i.
	\]
	Applying the same criterion to the stack,
	we conclude that
	\[
	C_{M_n(K)}(\mathcal A)=\mathcal A.
	\]
	Hence the stack is maximal commutative.
	
	Finally, the dimension formulas follow
	from
	\[
	\dim\mathcal A=1+ac+\sum_i d_i,
	\qquad
	\dim\mathcal A_i=1+ac+d_i.
	\]
\end{proof}

\begin{corollary}[Additivity of parameter types]
	\label{cor:stack-parameters}
	If two compatible parameter types are
	realizable, then
	\[
	(a,b_1,c;d_1)\star(a,b_2,c;d_2)
	=(a,b_1+b_2,c;d_1+d_2)
	\]
	is realizable.	 More generally,
	\[
	(a,b,c;d)^{\star q}=(a,qb,c;qd)
	\]
	for every \(q\ge1\).
\end{corollary}

\begin{remark}
	\label{rem:stack-associative}
	Up to the natural identification of the
	middle direct sums, stacking is associative
	and commutative. Accordingly, we use
	expressions such as
	\[
	\mathcal A_1^{\star r_1}
	\star\cdots\star
	\mathcal A_s^{\star r_s}
	\]
	without specifying parentheses or the order
	of the factors.
\end{remark}

\section{Standard and minimal equality bricks}
\label{sec:equality-bricks}

\subsection{Standard bricks}
\label{subsec:standard-bricks}

We begin with a uniform construction that attains
Laffey's signature-dependent bound for every pair
of outer dimensions \(a,c\ge2\).

Let
\[
U=K^a,\qquad Z=K^c,
\]
with standard bases \(e_1,\ldots,e_a\) and
\(f_1,\ldots,f_c\), respectively. Put
\[
W=K\oplus(K^c\otimes K^a),
\qquad \dim W=ac+1.
\]

For \(1\le i\le a\), define
\[
A_i(u)=(e_i^Tu,0),
\qquad
B_i(\lambda,x)=(I_c\otimes e_i^T)x.
\]
For \(1\le r\le c\), define
\[
\widetilde A_r(u)=(0,f_r\otimes u),
\qquad
\widetilde B_r(\lambda,x)=\lambda f_r.
\]
Thus \(A_i,\widetilde A_r:U\to W\) and
\(B_i,\widetilde B_r:W\to Z\).

Let \(\mathcal B_{a,c}\) be the algebra on
\(U\oplus W\oplus Z\) consisting of scalar
matrices and all matrices
\[
\begin{pmatrix}
	0&0&0\\
	A&0&0\\
	D&B&0
\end{pmatrix},
\qquad D\in\Hom(U,Z),
\]
where \((A,B)\) belongs to the span of the
\(a+c\) pairs
\[
(A_i,B_i),\qquad
(\widetilde A_r,\widetilde B_r).
\]

\begin{theorem}[Standard brick]
	\label{thm:standard-brick}
	For every \(a,c\ge2\), the algebra
	\(\mathcal B_{a,c}\) is local and maximal
	commutative, with cube-zero radical and
	parameter type
	\[
	B_{a,c}=(a,ac+1,c;a+c).
	\]
	For every \(q\ge1\), its \(q\)-fold stack
	has type
	\[
	B_{a,c}^{\star q}
	=
	(a,q(ac+1),c;q(a+c))
	\]
	and dimension
	\[
	\dim\mathcal B_{a,c}^{\star q}
	=1+ac+q(a+c).
	\]
	All these types attain Laffey's
	signature-dependent bound exactly.
\end{theorem}

\begin{proof}
	The defining products satisfy
	\[
	B_iA_j=0,\qquad
	\widetilde B_r\widetilde A_s=0,
	\]
	and
	\[
	B_i\widetilde A_r
	=
	\widetilde B_rA_i
	=
	f_re_i^T.
	\]
	Thus the prescribed generator pairs commute.
	
	The images of the \(A_i\) span the first
	summand \(K\) of \(W\), while those of the
	\(\widetilde A_r\) span \(K^c\otimes K^a\).
	Hence
	\[
	W=\sum_i\operatorname{im}A_i+
	\sum_r\operatorname{im}\widetilde A_r.
	\]
	Moreover, the products \(f_re_i^T\) form
	a basis of \(\Hom(U,Z)\). Their images
	therefore span \(Z\), giving the Loewy
	signature \((a,ac+1,c)\).
	
	The radical is cube-zero, and the algebra
	is local because it is of the form
	\(KI\oplus J\), with \(J\) nilpotent.
	
	It remains to prove maximal commutativity.
	By Proposition~\ref{prop:reduced-centralizer},
	we need only determine the reduced
	centralizer.
	
	Let \(P:U\to W\) and \(S:W\to Z\) satisfy
	\[
	SA_i=B_iP,\qquad
	S\widetilde A_r=\widetilde B_rP
	\]
	for every \(i,r\). Write
	\[
	P=\begin{pmatrix}p\\M\end{pmatrix},
	\qquad
	S=\begin{pmatrix}z&C\end{pmatrix},
	\]
	where
	\[
	p\in U^*,\quad
	M\in\Hom(U,K^c\otimes K^a),
	\]
	\[
	z\in Z,\quad
	C\in\Hom(K^c\otimes K^a,Z).
	\]
	
	The equations \(SA_i=B_iP\) give
	\[
	M=z\otimes I_a.
	\]
	Indeed, both sides evaluated at \(u\in U\)
	give
	\[
	z\,e_i^Tu=(I_c\otimes e_i^T)M(u),
	\]
	which determines every tensor coordinate
	of \(M(u)\).
	
	Similarly, the equations
	\(S\widetilde A_r=\widetilde B_rP\)
	give
	\[
	C=I_c\otimes p.
	\]
	Consequently, writing
	\[
	p=\sum_i p_ie_i^T,\qquad
	z=\sum_r z_rf_r,
	\]
	we obtain
	\[
	(P,S)=
	\sum_i p_i(A_i,B_i)
	+\sum_r z_r(\widetilde A_r,\widetilde B_r).
	\]
	These \(a+c\) pairs are linearly independent,
	so
	\[
	\dim C_{\mathrm{red}}(\mathcal B_{a,c})
	=a+c.
	\]
	The reduced-centralizer criterion proves
	that \(\mathcal B_{a,c}\) is maximal
	commutative.
	
	The stack theorem now gives
	\[
	B_{a,c}^{\star q}
	=(a,q(ac+1),c;q(a+c)).
	\]
	For this type,
	\[
	\frac{b(a+c)}{ac+1}
	=\frac{q(ac+1)(a+c)}{ac+1}
	=q(a+c)=d.
	\]
	Hence Laffey's bound is attained with
	exact equality.
\end{proof}

\begin{remark}[Relation to the construction $C_1$]
	\label{rem:C1}
	The standard bricks and their iterated stacks
	are instances of the earlier module-theoretic
	construction $C_1$ of Nowak-Morales
	\cite{NowakMorales1999}.
	
	Let $P=K^p$, $Z=K^q$, and
	\[
	B=KI_{P\oplus Z}\oplus\Hom(P,Z).
	\]
	Applying $C_1$ to the natural $B$-module
	$N=P\oplus Z$ with $r$ copies gives
	\[
	R=B\ltimes N^r,\qquad V=B^r\oplus N.
	\]
	The natural decomposition
	\[
	V=K^r\oplus
	\bigl(\Hom(P,Z)^r\oplus P\bigr)\oplus Z
	\]
	has middle layer
	\[
	\Hom(P,Z)^r\oplus P
	\cong
	\bigoplus_{i=1}^p
	\left(K\oplus(Z\otimes K^r)\right).
	\]
	
	Under this identification, the generators coming from the $i$-th coordinates of $P$ and $\Hom(P,Z)$ act precisely as the generator pairs defining $\mathcal B_{r,q}$. Products between distinct middle summands vanish, while the outer layers and the bottom-left block are shared.
	
	Consequently, the represented algebra obtained from $C_1$ is equivalent, after a change of basis, to the $p$-fold stack
	\[
	\mathcal B_{r,q}^{\star p}.
	\]
	In particular, $p=1$ recovers the standard bricks, while $p=q=r$ recovers the classical family of Laffey.
	
	The stacking operation developed here extends this construction by allowing arbitrary compatible cube-zero algebras, including primitive equality bricks, rather than only repeated standard bricks.
\end{remark}

Taking \(a=c=k\) and \(q=k\) gives
\[
B_{k,k}^{\star k}
=(k,k^3+k,k;2k^2),
\]
with
\[
n=k^3+3k,\qquad
\dim\mathcal B_{k,k}^{\star k}
=3k^2+1.
\]
These are the numerical parameters of
Laffey's classical equality family
\cite{Laffey1985}.

\subsection{Minimal equality bricks}
\label{subsec:minimal-equality-bricks}

The standard brick provides an exact equality type for every
outer pair $(a,c)$, with $a,c\ge2$. However, it need not be
the smallest integral solution of the unrounded equality
in Laffey's bound.

We distinguish two notions.

\begin{definition}[Sharp and exact equality types]
	\label{def:sharp-exact}
	A realizable parameter type $(a,b,c;d)$ is called
	\emph{sharp} if
	\[
	d=\left\lceil\frac{b(a+c)}{ac+1}\right\rceil.
	\]
	It is called an \emph{exact equality type} if
	\begin{equation}
		\label{eq:exact-equality}
		d(ac+1)=b(a+c).
	\end{equation}
\end{definition}

Every exact equality type is sharp, but the converse need
not hold. The distinction becomes relevant when stacking
sharp types, since the ceiling in Laffey's bound need not
be additive.

For fixed outer dimensions $(a,c)$, put
\begin{equation}
	\label{eq:minimal-equality-pair}
	g_{a,c}=\gcd(ac+1,a+c).
\end{equation}
The primitive positive integral solution of
\eqref{eq:exact-equality} is
\begin{equation}
	\label{eq:primitive-bd}
	b_0=\frac{ac+1}{g_{a,c}},
	\qquad
	d_0=\frac{a+c}{g_{a,c}}.
\end{equation}
Every positive integral solution is of the form
\[
(b,d)=q(b_0,d_0),
\qquad q\ge1.
\]

\begin{definition}[Minimal equality brick]
	\label{def:minimal-equality-brick}
	If the primitive exact equality type
	\[
	E_{a,c}=(a,b_0,c;d_0)
	\]
	is realizable, with $b_0,d_0$ given by
	\eqref{eq:primitive-bd}, we call it the
	\emph{minimal equality brick} for the outer
	pair $(a,c)$.
\end{definition}

The relation
\[
g_{a,c}b_0=ac+1,
\qquad
g_{a,c}d_0=a+c
\]
shows that whenever $E_{a,c}$ is realizable,
\begin{equation}
	\label{eq:primitive-root-standard}
	E_{a,c}^{\star g_{a,c}}
	\sim_{\mathrm{type}}B_{a,c}.
\end{equation}
Here $\sim_{\mathrm{type}}$ denotes equality
of parameter types, not necessarily an
isomorphism of algebras. Thus $E_{a,c}$ is
a primitive stack root of the standard brick
at the level of parameter types.

The arithmetic existence of $(b_0,d_0)$
does not guarantee a maximal commutative
realization. At exact equality, realizability
can be tested by surjectivity of the
reduced-centralizer map.

\begin{proposition}[Surjectivity criterion at exact equality]
	\label{prop:surjectivity-equality}
	Let $(A_i,B_i)$, $1\le i\le d$, be linearly
	independent generator pairs satisfying
	\[
	B_iA_j=B_jA_i
	\qquad(1\le i,j\le d),
	\]
	together with the spanning conditions
	\eqref{eq:A-images-span} and
	\eqref{eq:product-span}.
	Suppose that their parameter type $(a,b,c;d)$
	satisfies
	\[
	d(ac+1)=b(a+c).
	\]
	Then the associated cube-zero algebra containing
	the full block $\Hom(U,Z)$ is maximal commutative
	if and only if its reduced-centralizer map
	\[
	\Psi:
	\Hom(U,W)\oplus\Hom(W,Z)
	\longrightarrow\Hom(U,Z)^d
	\]
	is surjective.
\end{proposition}

\begin{proof}
	The exact equality relation gives
	\[
	\dim\operatorname{dom}\Psi
	-
	\dim\operatorname{codom}\Psi
	=
	b(a+c)-dac=d.
	\]
	The $d$ prescribed generator pairs belong
	to $\ker\Psi$ and are linearly independent,
	so $\dim\ker\Psi\ge d$.
	
	By Proposition~\ref{prop:reduced-centralizer},
	the algebra is maximal commutative precisely
	when $\dim\ker\Psi=d$. By rank--nullity, this
	holds if and only if
	\[
	\rank\Psi
	=
	\dim\operatorname{codom}\Psi,
	\]
	which is equivalent to surjectivity.
\end{proof}

The primitive arithmetic types are especially
simple for balanced and nearly balanced outer
dimensions.

\begin{proposition}[Balanced and nearly balanced arithmetic]
	\label{prop:balanced-nearly-balanced-arithmetic}
	For the balanced outer pair $(k,k)$,
	\[
	g_{k,k}=
	\begin{cases}
		1,&k\ \mathrm{even},\\
		2,&k\ \mathrm{odd}.
	\end{cases}
	\]
	Hence the primitive exact equality type is
	\[
	E_{k,k}^{\mathrm{arith}}=
	\begin{cases}
		(k,k^2+1,k;2k),
		&k\ \mathrm{even},\\[2mm]
		\left(k,\dfrac{k^2+1}{2},k;k\right),
		&k\ \mathrm{odd}.
	\end{cases}
	\]

	For the nearly balanced outer pair $(k,k+1)$,
	\[
	g_{k,k+1}=
	\begin{cases}
		3,&k\equiv1\pmod3,\\
		1,&k\not\equiv1\pmod3.
	\end{cases}
	\]
	Consequently,
	\[
	E_{k,k+1}^{\mathrm{arith}}=
	\begin{cases}
		\left(
		k,\dfrac{k^2+k+1}{3},k+1;
		\dfrac{2k+1}{3}
		\right),
		&k\equiv1\pmod3,\\[3mm]
		(k,k^2+k+1,k+1;2k+1),
		&k\not\equiv1\pmod3.
	\end{cases}
	\]
\end{proposition}

\begin{proof}
	For the balanced pair,
	\[
	g_{k,k}=\gcd(k^2+1,2k).
	\]
	The identity
	\[
	2(k^2+1)-k(2k)=2
	\]
	shows that $g_{k,k}$ divides $2$.
	If $k$ is even, then $k^2+1$ is odd,
	so $g_{k,k}=1$. If $k$ is odd,
	both arguments of the greatest common divisor
	are even, giving $g_{k,k}=2$.
	
	For the nearly balanced pair,
	\[
	g_{k,k+1}
	=\gcd(k^2+k+1,2k+1).
	\]
	Since
	\[
	4(k^2+k+1)-(2k+1)^2=3,
	\]
	we have $g_{k,k+1}\mid3$.
	Moreover, $3$ divides both arguments
	precisely when
	\[
	2k+1\equiv0\pmod3,
	\]
	equivalently $k\equiv1\pmod3$.
	Thus $g_{k,k+1}=3$ in this case
	and $g_{k,k+1}=1$ otherwise.
	
	Substitution into \eqref{eq:primitive-bd}
	gives the stated types.
\end{proof}

The distinction between arithmetic and
realizable primitive types is essential.
The standard bricks establish realizability
when $g_{a,c}=1$. For $g_{a,c}>1$, the existence
of a smaller equality brick requires a separate
construction. We address the entire balanced
branch next and then the first nontrivial
nearly balanced case $(4,5)$.

\subsection{The balanced branch}
\label{subsec:balanced-branch}

Proposition~\ref{prop:balanced-nearly-balanced-arithmetic} gives
$E_{k,k}^{\mathrm{arith}}=B_{k,k}$ for even $k$. For odd $k$,
the primitive exact equality type is
\[
E_{k,k}^{\mathrm{arith}}
=\left(k,\frac{k^2+1}{2},k;k\right),
\]
whose two-fold stack has the type $B_{k,k}$. We prove the
realizability of every such odd balanced root in characteristic zero.

\subsubsection{Matrix systems and transverse certificates}
\label{subsubsec:balanced-matrix-data}
Let $V=Z=K^k$, $\dim W=b=(k^2+1)/2$, where $\operatorname{char}K=0$.
Suppose that injectively linearly parameterized pairs
\[
A_x:V\longrightarrow W,\qquad B_x:W\longrightarrow Z\quad(x\in V)
\]
satisfy $B_xA_y=B_yA_x$, and that the multiplication map
\[
\mu:V\otimes V\longrightarrow W,\qquad x\otimes y\longmapsto A_x y
\]
is onto and symmetric. All relations used below are expressed through this multiplication map and its kernel. For $w\in W$ put
\[
C_w:V\longrightarrow Z,\qquad C_w(x)=B_xw.
\]
The \emph{old rigidity condition} is
\begin{equation}\label{oldrig}
	\{(P,S): S A_x=B_x P\ \text{for all }x\in V\}
	=\{(A_r,B_r):r\in V\}.
\end{equation}
Here $P:V\to W$ and $S:W\to Z$. Equivalently, a map $P:V\to W$ for which the bilinear map $(x,y)\mapsto B_xP(y)$ vanishes on $\ker\mu$ must equal $A_r$ for some $r\in V$.

Fix $w_0\in W$ and $\varphi\in W^*$. Write
\[
M=C_{w_0}\in\Hom(V,Z),\qquad
G(x,y)=\varphi(A_xy),
\]
and assume that $M$ and the symmetric matrix $G$ are invertible and $\varphi(w_0)\ne0$. Finally, let
\[
\mathcal E=\{M^{-1}C_w:w\in W\}\subset\End(V),\qquad
\mathcal S_G=\{T\in\End(V): (x,y)\mapsto G(x,Ty)\text{ factors through }\mu\}.
\]
The additional \emph{transversality condition} is
\begin{equation}\label{trans}
	\mathcal E\cap\mathcal S_G=KI_V.
\end{equation}
We require $w\mapsto C_w$ to be injective. This condition is verified for the base brick and is preserved by the extension below.

\subsubsection{The initial brick $E_{3,3}$}
\label{subsubsec:E33}

Work with the ordered bases $V=Z=K^3$ and
$W=\langle x_0^2,x_0x_1,x_0x_2,x_1^2,x_2^2\rangle$.
The product in $W$ is the symmetric quadratic product modulo $x_1x_2=0$.
The three multiplication maps are
\[
A_0=\begin{pmatrix}1&0&0\\0&1&0\\0&0&1\\0&0&0\\0&0&0\end{pmatrix},\quad
A_1=\begin{pmatrix}0&0&0\\1&0&0\\0&0&0\\0&1&0\\0&0&0\end{pmatrix},\quad
A_2=\begin{pmatrix}0&0&0\\0&0&0\\1&0&0\\0&0&0\\0&0&1\end{pmatrix}.
\]
Set
\[
B_0=\begin{pmatrix}0&0&1&1&0\\0&0&0&0&0\\0&1&0&0&1\end{pmatrix},\quad
B_1=\begin{pmatrix}0&1&0&0&0\\0&0&0&1&0\\1&0&0&0&0\end{pmatrix},\quad
B_2=\begin{pmatrix}1&0&0&0&0\\0&0&0&0&1\\0&0&1&0&0\end{pmatrix}.
\]
Choose
$w_0=x_0x_2+x_1^2$ and
$\varphi(x_0x_2)=\varphi(x_1^2)=1$, zero on the other basis monomials.
Thus
\[
M=C_{w_0}=\operatorname{diag}(2,1,1),\qquad
G=(\varphi(A_i x_j))_{ij}=\begin{pmatrix}0&0&1\\0&1&0\\1&0&0\end{pmatrix},
\quad \varphi(w_0)=2.
\]
The three pairs satisfy $B_iA_j=B_jA_i$. To check rigidity,
write $P=(p_0,p_1,p_2):V\to W$. The kernel relations for the
symmetric multiplication map are equivalent to
\[
B_0p_1=B_1p_0,\quad B_0p_2=B_2p_0,\quad
B_1p_2=0,\quad B_2p_1=0.
\]
Ordinary elimination gives
$P=sA_0+tA_1+uA_2$; the corresponding $S$ is uniquely
$sB_0+tB_1+uB_2$ because the $A_i$-images span $W$.
For the certificate, put $T_w=M^{-1}C_w$ and let $\mathcal S_G$
consist of maps $T$ for which $(x,y)\mapsto G(x,Ty)$ factors through
symmetric multiplication modulo $x_1x_2=0$. Writing
$w=(u_0,u_1,u_2,u_3,u_4)^T$ gives
\[
T_w=\begin{pmatrix}
	(u_2+u_3)/2&u_1/2&u_0/2\\
	0&u_3&u_4\\u_1+u_4&u_0&u_2
\end{pmatrix}.
\]
Its associated bilinear form is symmetric and zero on $x_1\otimes x_2$
if and only if $u_0=u_1=u_4=0$ and $u_2=u_3$.
Thus $T_w\in\mathcal S_G$ precisely for $w\in K w_0$.
The map $w\mapsto C_w$ is injective (as is also clear from this
matrix), $M,G$ are invertible, and $\varphi(w_0)=2\ne0$.
This provides the initial transverse certificate.

\subsubsection{The two-variable block extension}\label{subsubsec:transverse-extension}
Introduce $p,q$ and basis vectors $\alpha,\beta$, and put
\[
V^+=V\oplus K p\oplus K q,\quad Z^+=Z\oplus K\alpha\oplus K\beta,\quad
W^+=W\oplus pV\oplus qV\oplus K p^2\oplus K q^2.
\]
Thus $\dim W^+=b+2k+2=((k+2)^2+1)/2$. Use the ordered decompositions shown above. For $x\in V$, define the old-index pair by
\begin{equation}\label{Aold}
	A_x^+=\begin{pmatrix}
		A_x&0&0\\0&x&0\\0&0&x\\0&0&0\\0&0&0
	\end{pmatrix},\qquad
	B_x^+=\begin{pmatrix}
		B_x&0&0&Mx&Mx\\
		0&0&G(x,\cdot)&0&0\\
		0&G(x,\cdot)&0&0&0
	\end{pmatrix}.
\end{equation}
In the $A$-matrix, $x$ denotes the column of $x\in V$; $G(x,\cdot)$ is a row vector. For the new indices set
\begin{equation}\label{Anew}
	A_p^+=\begin{pmatrix}
		0&0&0\\I_k&0&0\\0&0&0\\0&1&0\\0&0&0
	\end{pmatrix},\qquad
	B_p^+=\begin{pmatrix}
		0&M&0&0&0\\
		0&0&0&1&0\\
		\varphi&0&0&0&0
	\end{pmatrix},
\end{equation}
\begin{equation}\label{Anew2}
	A_q^+=\begin{pmatrix}
		0&0&0\\0&0&0\\I_k&0&0\\0&0&0\\0&0&1
	\end{pmatrix},\qquad
	B_q^+=\begin{pmatrix}
		0&0&M&0&0\\
		\varphi&0&0&0&0\\
		0&0&0&0&1
	\end{pmatrix}.
\end{equation}
These formulas define the enlarged pairs directly; no cubic polynomial map is required. Using $B_xA_y=B_yA_x$, symmetry of $G$, and $C_{w_0}=M$, direct multiplication shows
\begin{equation}\label{comm}
	B_u^+ A_v^+=B_v^+ A_u^+\qquad(u,v\in V^+).
\end{equation}
In particular $A_p^+q=A_q^+p=0$.

\paragraph{Rigidity of the extension.}
\begin{proposition}[Conditional block extension]\label{prop:certified-extension}
	Under assumptions \eqref{oldrig}--\eqref{trans}, the enlarged system is rigid: if
	$P^+:V^+\to W^+$ and $S^+:W^+\to Z^+$ satisfy
	\[
	S^+A_u^+=B_u^+P^+\qquad(u\in V^+),
	\]
	then there exists a unique $v\in V^+$ such that $P^+=A_v^+$ and $S^+=B_v^+$.
\end{proposition}
\begin{proof}

	Since the symmetric multiplication map
	$\mu^+:V^+\otimes V^+\to W^+$ is surjective,
	the equation
	$S^+A_u^+=B_u^+P^+$
	admits a unique solution $S^+$ precisely when
	the map
	\[
	u\otimes v\longmapsto B_u^+P^+(v)
	\]
	vanishes on $\ker\mu^+$.
	We therefore determine all $P^+$ satisfying
	these compatibility relations.
	
 For $x\in V$, write
	\[
	P^+(x)=P_0(x)+pU(x)+qT(x)+f(x)p^2+g(x)q^2,
	\]
	where $P_0:V\to W$, $U,T\in\End(V)$, and $f,g\in V^*$. Write
	\[
	\begin{aligned}
		P^+(p)&=w_p+pa+qb+cp^2+dq^2,\\
		P^+(q)&=w_q+pe+qh+c'p^2+d'q^2.
	\end{aligned}
	\]
	The relation $pq=0$ forces $B_p^+P^+(q)=B_q^+P^+(p)=0$. Reading off the three $Z^+$ components gives
	\begin{equation}\label{first}
		e=b=0,\qquad c'=d=0,\qquad \varphi(w_p)=\varphi(w_q)=0.
	\end{equation}
	Next the relations between $x\in V$ and $p,q$ yield
	\begin{equation}\label{second}
		f=g=0,\qquad U=M^{-1}C_{w_p}+cI_V,\qquad
		T=M^{-1}C_{w_q}+d'I_V,
	\end{equation}
	and
	\begin{equation}\label{third}
		\varphi(P_0(x))=G(x,a)=G(x,h)\qquad(x\in V).
	\end{equation}
	The old--old relations, projected to $Z$ and to the two new coordinates of $Z^+$, say respectively
	\begin{equation}\label{fourth}
		P_0=A_r\text{ for some }r\in V,\qquad U,T\in\mathcal S_G.
	\end{equation}
	The parameter $r$ is unique because $v\mapsto A_v$ is injective. From \eqref{second} and \eqref{fourth} we obtain
	\[
	M^{-1}C_{w_p},\ M^{-1}C_{w_q}\in\mathcal E\cap\mathcal S_G=KI_V.
	\]
	Since $C_{w_0}=M$ and the $C$-map is injective, $w_p,w_q\in K w_0$. Together with \eqref{first} and $\varphi(w_0)\ne0$ this yields $w_p=w_q=0$. Hence $U=cI$, $T=d'I$. Since $P_0=A_r$, equation \eqref{third} and the nondegeneracy of $G$ give $a=h=r$.
Thus
\[
P^+=A_{r+cp+d'q}^+.
\]
By uniqueness, the corresponding map is
$S^+=B_{r+cp+d'q}^+$.
The parameter $r+cp+d'q$ is unique
because $v\mapsto A_v^+$ is injective.
\end{proof}

\paragraph{Persistence of transversality.}
We retain the hypotheses of Proposition~\ref{prop:certified-extension}. In particular, the enlarged system is rigid. The following argument makes the persistence step explicit in ordinary block notation. 

For $\xi=w+pa+qb+cp^2+dq^2\in W^+$, direct multiplication of the matrices in \eqref{Aold}--\eqref{Anew2} gives
\begin{equation}\label{Cplus}
	C^+_\xi=
	\begin{pmatrix}
		C_w+(c+d)M & Ma & Mb\\
		G(b,\cdot)&c&\varphi(w)\\
		G(a,\cdot)&\varphi(w)&d
	\end{pmatrix}.
\end{equation}
Here $C_\xi^+(u)=B_u^+\xi$. If $C_\xi^+=0$, the two upper-right blocks give $a=b=0$; the lower-right block gives $c=d=\varphi(w)=0$, and the remaining upper-left block gives $C_w=0$. Injectivity of the old $C$-map implies $w=0$. Hence $\xi\mapsto C_\xi^+$ is injective.

Set $w_0^+=p^2+q^2$. Then
\[
C_{w_0^+}^+=\operatorname{diag}(2M,1,1)
\]
is invertible. Choose $r\in V$ with $G(r,r)\ne0$; one exists by nondegeneracy of $G$ and $\operatorname{char}K\ne2$. Define $\Phi_t\in (W^+)^*$ by
\begin{align*}
	\Phi_t|_W&=\varphi,& \Phi_t(px)&=tG(r,x),& \Phi_t(qx)&=0,\\
	\Phi_t(p^2)&=1,&\Phi_t(q^2)&=1.
\end{align*}
The corresponding symmetric form on $V^+=V\oplus Kp\oplus Kq$ has matrix
\begin{equation}\label{Gplus}
	G_t=\begin{pmatrix}G&tGr&0\\t r^T G&1&0\\0&0&1\end{pmatrix},\qquad
	\det G_t=\det G\,(1-t^2G(r,r)).
\end{equation}
Moreover $\Phi_t(w_0^+)=2\ne0$.

\begin{proposition}[Persistence of transversality]\label{prop:transversality-persistence}
	For some $t\in K^\times$, the pair $(w_0^+,\Phi_t)$ is a transverse certificate for the enlarged system. Consequently Proposition~\ref{prop:certified-extension} may be iterated.
\end{proposition}
\begin{proof}
	Since $K$ has characteristic zero, it is infinite. We show that the relevant intersection is scalar for all but finitely many $t$. Write
	\[
	\widehat T_\xi=(C_{w_0^+}^+)^{-1}C_\xi^+
	=\begin{pmatrix}
		\frac12T_w+\frac{c+d}{2}I & \frac12a&\frac12b\\
		G(b,\cdot)&c&\varphi(w)\\
		G(a,\cdot)&\varphi(w)&d
	\end{pmatrix},\qquad T_w=M^{-1}C_w.
	\]
	Suppose $\widehat T_\xi\in\mathcal S_{G_t}$, i.e. the bilinear form
	$H_\xi(u,v)=G_t(u,\widehat T_\xi v)$ factors through the enlarged symmetric multiplication map $\mu^+$. It is symmetric and vanishes on $p\otimes q$ because $pq=0$.
	
	The $p,q$ entries yield
	\begin{equation}\label{eq27}
		G(r,b)=0,\qquad \varphi(w)=0 \qquad(t\ne0).
	\end{equation}
	The $(x,q)$ entries for $x\in V$ give $a=b/2$. The restriction to $V\otimes V$ gives
	\begin{equation}\label{eq29}
		\frac12 T_w+t R_b\in\mathcal S_G,\qquad
		R_b(y)=G(b,y)r.
	\end{equation}
	To solve this, note that $\dim\mathcal E=\dim\mathcal S_G=\dim W=(k^2+1)/2$ and $\mathcal E\cap\mathcal S_G=KI$. Hence $\mathcal E+\mathcal S_G=\End(V)$. Since $T_{w_0}=I$ and $\varphi(w_0)\ne0$, the map
	\[
	Q:\ker\varphi\longrightarrow\End(V)/\mathcal S_G,\qquad Q(w)=[T_w]
	\]
	is injective and both sides have dimension $(k^2-1)/2$, so it is an isomorphism. There is thus a unique linear map $L:V\to\ker\varphi$ satisfying
	\[
	[T_{L(b)}]=-2[R_b].
	\]
	Equations \eqref{eq27}--\eqref{eq29} imply
	\begin{equation}\label{eq33}
		w=tL(b).
	\end{equation}
	
	For the remaining $(x,p)$ symmetry, put
	\[
	A=\frac12T_w+\frac{c+d}{2}I,\qquad \lambda=\frac{c-d}{2}.
	\]
	Using $a=b/2$ and denoting by $T_w^*$ the $G$-adjoint, comparison gives
	\begin{equation}\label{eq34}
		b=\frac{4t}{3}\left(\lambda r-\frac12 T_w^*r\right).
	\end{equation}
	Substitute \eqref{eq33}, pair with $r$, and use $G(r,b)=0$ and $G(r,r)\ne0$ to obtain
	\[
	\lambda=\frac{t}{2G(r,r)}G(r,T_{L(b)}^*r).
	\]
	Therefore
	\begin{equation}\label{eq36}
		(I-t^2\mathcal K)b=0,\qquad
		\mathcal K(b)=\frac23\left(
		\frac{G(r,T_{L(b)}^*r)}{G(r,r)}r-T_{L(b)}^*r
		\right).
	\end{equation}
	The map $\mathcal K$ is linear and independent of $t$. The polynomial $\det(I-t^2\mathcal K)$ has constant term $1$, while \eqref{Gplus} shows that $\det G_t$ is a nonzero polynomial. Thus one can choose $t\ne0$ for which both determinants are nonzero. Then \eqref{eq36} gives $b=0$, hence $a=0$ and $w=0$. Equation \eqref{eq34} gives $\lambda=0$, hence $c=d$. So $\xi=cw_0^+$, and therefore
	\[
	\mathcal E_{w_0^+}\cap\mathcal S_{G_t}=K I_{V^+}.
	\]
	Together with invertibility of $C_{w_0^+}^+$, nondegeneracy of $G_t$, $\Phi_t(w_0^+)\ne0$, and injectivity of $C^+$, this proves persistence.
\end{proof}

\subsubsection{Existence of all balanced roots}

\begin{theorem}[Balanced minimal equality bricks]
	\label{thm:balanced-minimal-bricks}
	Let $K$ be a field of characteristic zero. For every integer $k\ge2$,
	the minimal balanced equality brick exists, with type
	\[
	E_{k,k}=\begin{cases}
		(k,k^2+1,k;2k),& k\text{ even},\\
		\bigl(k,(k^2+1)/2,k;k\bigr),& k\text{ odd}.
	\end{cases}
	\]
	For odd $k$, the two-fold stack of $E_{k,k}$ has the same
	parameter type as the standard brick $B_{k,k}$.
\end{theorem}
\begin{proof}
	If $k$ is even, the assertion follows from
	Theorem~\ref{thm:standard-brick}. For $k=3$, the displayed
	matrices give a rigid system with a transverse certificate.
	Propositions~\ref{prop:certified-extension} and
	\ref{prop:transversality-persistence} allow induction $k\mapsto k+2$.
	The resulting pairs satisfy the Loewy spanning conditions and
	have $d=k$ linearly independent generators, so
	Proposition~\ref{prop:reduced-centralizer} gives a local maximal
	commutative algebra of type
	$\bigl(k,(k^2+1)/2,k;k\bigr)$.
	The exact equality relation
	$k(k^2+1)=((k^2+1)/2)(2k)$ proves sharpness, and
	Theorem~\ref{thm:stack} gives the assertion about two-fold stacks.
\end{proof}

\begin{remark}[The case $E_{5,5}$ in arbitrary characteristic]
	\label{rem:E55-all-characteristic}
	The general induction uses characteristic zero. The stronger claim
	that $E_{5,5}$ exists over every field relies on the separate
	integral model and unimodular certificate in the Supplementary
	Material, not on the present induction.
\end{remark}

\section{The first nearly balanced root $(4,5)$}
\label{sec:nearly-balanced}

For the nearly balanced outer pair $(k,k+1)$,
Proposition~\ref{prop:balanced-nearly-balanced-arithmetic}
shows that the primitive arithmetic equality type
is smaller than the standard brick precisely when
$k\equiv1\pmod3$. In this case,
\[
g_{k,k+1}=3,
\]
and the primitive arithmetic type is
\[
E_{k,k+1}^{\mathrm{arith}}=
\left( k,\frac{k^2+k+1}{3},k+1; \frac{2k+1}{3}\right).
\]
Its three-fold stack has the parameter type
$B_{k,k+1}$, provided that the primitive type
is realizable.

The first such pair with $k\ge2$ is $(4,5)$.
We establish the realizability of
\[
E_{4,5}=(4,7,5;3)
\]
and, more generally, the sharpness of Laffey's
bound throughout the $(4,5)$-sector for $b\ge4$.

The construction of $E_{4,5}$ is valid over
arbitrary fields. A characteristic-independent
verification, combining two integral models
with explicit certificates in the exceptional
characteristics $2,5,7,$ and $11$, is provided
in the Supplementary Material.

\subsection{Sharp seed bricks}
\label{subsec:45-seeds}

\begin{proposition}[Sharp bricks for $(4,5)$] \label{prop:E45}
	Over every field of characteristic zero,
	the following four parameter types are
	realizable:
	\[
	(4,4,5;2),\quad (4,6,5;3),\quad
	(4,7,5;3),\quad (4,9,5;4).
	\]
	All four types are sharp.
	Moreover, the types $(4,4,5;2)$ and
	$(4,7,5;3)$ are realizable over every field.
	In particular, the minimal equality brick
	$E_{4,5}=(4,7,5;3)$ exists over every field.
\end{proposition}

\begin{proof}
	For the outer pair $(4,5)$, Laffey's bound
	reduces to
	\[
	d\ge
	\left\lceil\frac{9b}{21}\right\rceil
	=
	\left\lceil\frac{3b}{7}\right\rceil.
	\]
	For $b=4,6,7,9$, the corresponding minimal
	values are $d=2,3,3,4$, respectively.
	
	The four types are realized by explicit
	integral generator pairs recorded in the
	Supplementary Material. Their reduced-centralizer
	certificates have the following dimensions:
	\[
	\begin{array}{c|c|c|c}
		\text{type}
		& \text{matrix of }\Psi
		& \rank\Psi
		& \dim\ker\Psi\\ \hline
		(4,4,5;2)&40\times36&34&2\\
		(4,6,5;3)&60\times54&51&3\\
		(4,7,5;3)&60\times63&60&3\\
		(4,9,5;4)&80\times81&77&4
	\end{array}
	\]
	Thus, in each case, the reduced centralizer
	has dimension $d$, and
	Proposition~\ref{prop:reduced-centralizer}
	gives maximal commutativity.
	
	The integral generator systems and the exact
	reduced-centralizer rank certificates establish
	all four realizations in characteristic zero.
	Characteristic-independent certificates are
	available for $(4,4,5;2)$ and $(4,7,5;3)$,
	as documented in the Supplementary Material.

	The Supplementary Material also  verifies
	commutativity, linear independence, and the
	Loewy spanning conditions.
	
	Finally,
	\[
	3(4\cdot5+1)=7(4+5),\qquad \gcd(21,9)=3. 
	\]
	Hence $(4,7,5;3)$ is the primitive exact
	equality type, and therefore realizes
	$E_{4,5}$.
\end{proof}

\subsection{One-dimensional middle extension}
\label{subsec:middle-extension}

The remaining residue classes can be obtained
by a simple extension that preserves maximal
commutativity.

\begin{lemma}[One-dimensional middle extension]
	\label{lem:middle-extension}
	If $(a,b,c;d)$ is realizable over a field $K$,
	then so is
	\[
	(a,b+r,c;d+r)
	\]
	for every integer $r\ge0$.
\end{lemma}

\begin{proof}
	It suffices to consider $r=1$.
	Let $U\oplus W\oplus Z$ be a Loewy realization
	of the original algebra, with generator pairs
	$(A_i,B_i)$, $1\le i\le d$.
	
	Replace $W$ by $\widetilde W=W\oplus K$
	and extend the old pairs by zero:
	\[
	\widetilde A_i= \begin{pmatrix}A_i\\0\end{pmatrix},
	\qquad \widetilde B_i=\begin{pmatrix}B_i&0\end{pmatrix}.
	\]
	
	Choose nonzero linear maps
	$\alpha:U\to K$ and $\beta:K\to Z$.
	Add the generator pair
	\[
	A_0=\begin{pmatrix}0\\\alpha\end{pmatrix},
	\qquad	B_0=\begin{pmatrix}0&\beta\end{pmatrix}.
	\]
	All mixed products vanish, so the enlarged
	system is commutative. Its $A$-images span
	$\widetilde W$, and the old products still
	generate $Z$. Thus the Loewy signature
	is $(a,b+1,c)$.
	
	To prove maximality, write a reduced-centralizer
	pair as
	\[
	P=\begin{pmatrix}P'\\p\end{pmatrix},
	\qquad 	S=\begin{pmatrix}S'&s\end{pmatrix},
	\]
	where $p:U\to K$ and $s:K\to Z$.
	
	The equations corresponding to the old
	generators are
	\[
	S'A_i=B_iP', 
	\]
	so $(P',S')$ belongs to the original
	reduced centralizer, of dimension $d$.
	
	The equation for the new generator is
	\[
	s\alpha=\beta p.
	\]
	Since $\alpha$ and $\beta$ are nonzero,
	equality of these rank-one operators implies
	\[
	p=\lambda\alpha,\qquad	s=\lambda\beta
	\]
	for a unique $\lambda\in K$.
	Therefore
	\[
	\dim C_{\mathrm{red}}(\widetilde A)=d+1.
	\]
	The enlarged system has $d+1$ independent
	generator pairs. By the reduced-centralizer
	criterion, it is maximal commutative and
	has parameter type $(a,b+1,c;d+1)$.
	
	Iteration proves the assertion for all $r$.
\end{proof}

\subsection{The complete $(4,5)$-sector}
\label{subsec:45-sector}

\begin{corollary}[Sharp residue representatives]
	\label{cor:45-residues}
	For every $r\in\{4,5,6,7,8,9,10\}$,
	the type
	\[
	T_r= \left(4,r,5; \left\lceil\frac{3r}{7}\right\rceil\right)
	\]
	is realizable over every field of characteristic zero.
\end{corollary}

\begin{proof}
	Proposition~\ref{prop:E45} supplies the
	types with $r=4,6,7,9$.
	Applying Lemma~\ref{lem:middle-extension}
	once to the types with $r=4,7,9$
	gives, respectively,
	\[
	(4,5,5;3),\quad	(4,8,5;4),\quad (4,10,5;5).
	\]
	All seven values of $d$ agree with
	$\lceil3r/7\rceil$, proving sharpness.
\end{proof}

\begin{theorem}[Complete $(4,5)$-sector]
	\label{thm:45-complete-ray}
	For every field $K$ of characteristic zero and every integer $b\ge4$,
	there exists a local maximal commutative
	$K$-subalgebra with cube-zero radical and
	parameter type
	\[
	\left( 	4,b,5; 	\left\lceil\frac{3b}{7}\right\rceil \right).
	\]
	Consequently, Laffey's signature-dependent
	bound is attained for every Loewy signature
	$(4,b,5)$ with $b\ge4$.
\end{theorem}

\begin{proof}
	Write
	\[
	b=r+7q,\qquad q\ge0,\qquad 4\le r\le10.
	\]
	By Corollary~\ref{cor:45-residues},
	there is a sharp realizable type
	\[
	T_r=(4,r,5;d_r),\qquad	d_r=\left\lceil\frac{3r}{7}\right\rceil.
	\]
	The primitive equality brick $E_{4,5}=(4,7,5;3)$ is realizable by
	Proposition~\ref{prop:E45}.
	
	The stack theorem therefore gives
	\[
	T_r\star E_{4,5}^{\star q}=(4,r+7q,5;d_r+3q).
	\]
	Since
	\[
	d_r+3q=\left\lceil\frac{3(r+7q)}{7}\right\rceil
	=\left\lceil\frac{3b}{7}\right\rceil,
	\]
	the resulting type is sharp.
	All constructions are valid over $K$.
\end{proof}

\section{The outer-pair envelope}
\label{sec:outer-envelope}

We now minimize Laffey's signature-dependent
bound over all Loewy signatures of a fixed
matrix size $n$.

For $a+c<n$, put $b=n-a-c$ and recall that
\[
L_3(n)=
\min_{\substack{a,b,c\ge1\\a+b+c=n}}
\left\{
1+ac+
\left\lceil\frac{b(a+c)}{ac+1}\right\rceil
\right\}.
\]
Since $b(a+c)/(ac+1)=b-\beta_{a,c}b$,
where
\[
\beta_{a,c}=
\frac{(a-1)(c-1)}{ac+1},
\]
the expression to be minimized becomes
\[
n-\left\lfloor\gamma_{a,c}(n)\right\rfloor,
\]
with
\begin{equation}
	\label{eq:outer-gamma}
	\gamma_{a,c}(n)=
	\frac{(a-1)(c-1)}{ac+1}
	\bigl(n-(a+1)(c+1)\bigr).
\end{equation}
Thus
\begin{equation}
	\label{eq:L3-envelope}
	L_3(n)=
	n-
	\max_{\substack{a,c\ge1\\a+c<n}}
	\left\lfloor\gamma_{a,c}(n)\right\rfloor.
\end{equation}

For $a=1$ or $c=1$, the function
$\gamma_{a,c}$ vanishes. For $n\ge14$,
the pair $(2,2)$ gives a positive value,
so we may restrict the continuous
maximization to $a,c\ge2$.

\begin{theorem}[Balanced--nearly-balanced envelope]
	\label{thm:outer-envelope}
	For every integer $n\ge14$, the maximum of
	$\gamma_{a,c}(n)$ over admissible integer
	outer pairs is attained by a balanced or
	nearly balanced pair.
	
	More precisely, the pair $(k,k)$ maximizes
	the continuous profile for
	\begin{equation}
		\label{eq:balanced-envelope-range}
		k^3-k^2+3k-1
		\le n\le
		k^3+k^2+3k+1,
	\end{equation}
	whereas $(k,k+1)$ maximizes it for
	\begin{equation}
		\label{eq:nearly-balanced-envelope-range}
		k^3+k^2+3k+1
		\le n\le
		k^3+2k^2+4k+2.
	\end{equation}
	Here $k\ge2$, and neighboring pairs
	have equal continuous values at the
	common endpoints.
	
	Consequently,
	\begin{equation}
		\label{eq:L3-balanced-reduction}
		L_3(n)=
		n-
		\left\lfloor
		\max_{\substack{k\ge2\\
				(k,k)\text{ or }(k,k+1)\text{ admissible}}}
		\left\{
		\gamma_{k,k}(n),
		\gamma_{k,k+1}(n)
		\right\}
		\right\rfloor.
	\end{equation}
\end{theorem}

\begin{proof}
	By symmetry, suppose $2\le a\le c$.
	We first show that a pair with
	$c-a\ge2$ can be replaced by a pair
	closer to the diagonal without
	decreasing $\gamma$.
	
	Direct subtraction from
	\eqref{eq:outer-gamma} gives
	\[
	\gamma_{a+1,c}(n)-\gamma_{a,c}(n)
	=
	\frac{(c^2-1)
		\bigl(n-N_+(a,c)\bigr)}
	{(ac+1)(ac+c+1)},
	\]
	where
	\[
	N_+(a,c)=
	a^2c+ac+2a+c+1.
	\]
	Similarly,
	\[
	\gamma_{a,c-1}(n)-\gamma_{a,c}(n)
	=
	\frac{(a^2-1)
		\bigl(N_-(a,c)-n\bigr)}
	{(ac+1)(ac-a+1)},
	\]
	where
	\[
	N_-(a,c)=
	ac^2-ac+a+2c-1.
	\]
	The thresholds satisfy
	\[
	N_-(a,c)-N_+(a,c)
	=(ac+1)(c-a-2)\ge0.
	\]
	Hence either $n\ge N_+(a,c)$, in
	which case increasing $a$ does not
	decrease $\gamma$, or
	$n\le N_-(a,c)$, in which case
	decreasing $c$ does not decrease it.
	
	Both operations reduce the difference
	$c-a$. If increasing $a$ is inadmissible
	because $a+c=n-1$, decreasing $c$ is
	admissible and improves the value.
	Iteration therefore produces a balanced
	or nearly balanced pair with at least
	the original value.
	
	It remains to compare these pairs.
	For $k\ge2$, direct calculation gives
	\begin{equation}
		\label{eq:gamma-balanced-near-difference}
		\gamma_{k,k+1}(n)-\gamma_{k,k}(n)
		=
		\frac{(k^2-1)
			\bigl(n-k^3-k^2-3k-1\bigr)}
		{(k^2+1)(k^2+k+1)}.
	\end{equation}
	Thus $(k,k)$ dominates $(k,k+1)$
	precisely when
	\[
	n\le k^3+k^2+3k+1.
	\]
	
	Likewise,
	\begin{equation}
		\label{eq:gamma-near-next-difference}
		\gamma_{k+1,k+1}(n)-\gamma_{k,k+1}(n)
		=
		\frac{k(k+2)
			\bigl(n-k^3-2k^2-4k-2\bigr)}
		{(k^2+k+1)(k^2+2k+2)}.
	\end{equation}
	Hence $(k,k+1)$ dominates
	$(k+1,k+1)$ precisely when
	\[
	n\le k^3+2k^2+4k+2.
	\]
	
	The transition values are strictly
	increasing with $k$, and
	\[
	k^3+2k^2+4k+2
	=
	(k+1)^3-(k+1)^2+3(k+1)-1.
	\]
	Consequently, the intervals
	\eqref{eq:balanced-envelope-range}
	and \eqref{eq:nearly-balanced-envelope-range}
	form consecutive segments covering
	all $n\ge14$.
	
	The two comparison formulas show that
	each pair dominates its neighbors
	between its adjacent transition values.
	Since the transition values are ordered,
	it also dominates all earlier and later
	pairs in the balanced--nearly-balanced
	sequence. Together with the preceding
	reduction, this proves global maximality.
	
	Finally, taking integer parts gives
	\eqref{eq:L3-balanced-reduction}.
\end{proof}

\begin{remark}[Continuous and discrete ties]
	\label{rem:envelope-ties}
	The continuous maximizer need not be the
	only outer pair minimizing the numerical
	profile $L_3(n)$. Because $L_3(n)$ depends
	on $\lfloor\gamma_{a,c}(n)\rfloor$, a pair
	with a slightly smaller continuous value
	may yield the same discrete minimum.
	
	This distinction is relevant to the
	attainment constructions below: a sharp
	algebra may realize $L_3(n)$ even when
	its outer pair does not maximize the
	continuous envelope.
\end{remark}

\subsection{Arithmetic propagation and attainment}
\label{subsec:arithmetic-attainment}

For fixed outer dimensions \(a,c\), stacking is
additive in \(b\) and \(d\), but sharpness is
affected by the ceiling in Laffey's bound.
Recall that
\[
\beta_{a,c}
=
\frac{(a-1)(c-1)}{ac+1}.
\]
Since
\[
\frac{a+c}{ac+1}=1-\beta_{a,c},
\]
a sharp type \(T=(a,b,c;d)\) satisfies
\[
d
=
\left\lceil
\frac{b(a+c)}{ac+1}
\right\rceil
=
b-\lfloor\beta_{a,c}b\rfloor.
\]
Put
\[
\phi_{a,c}(b)=\{\beta_{a,c}b\},
\]
where \(\{x\}\) denotes the fractional part
of \(x\).

\begin{proposition}[Sharp stacking criterion]
	\label{prop:sharp-stacking}
	Let
	\[
	T_i=(a,b_i,c;d_i),
	\qquad i=1,\ldots,r,
	\]
	be sharp realizable types with the same outer
	dimensions. Then their stack is sharp if and
	only if
	\[
	\sum_{i=1}^r\phi_{a,c}(b_i)<1.
	\]
\end{proposition}

\begin{proof}
	By the stack theorem, the resulting type has
	\[
	b=\sum_i b_i,
	\qquad d=\sum_i d_i.
	\]
	Since each \(T_i\) is sharp,
	\[
	d
	=
	\sum_i b_i
	-
	\sum_i\lfloor\beta_{a,c}b_i\rfloor.
	\]
	The stack is sharp precisely when
	\[
	\sum_i\lfloor\beta_{a,c}b_i\rfloor
	=
	\left\lfloor
	\beta_{a,c}\sum_i b_i
	\right\rfloor.
	\]
	Using
	\[
	\left\lfloor\sum_i x_i\right\rfloor
	=
	\sum_i\lfloor x_i\rfloor
	+
	\left\lfloor\sum_i\{x_i\}\right\rfloor
	\]
	gives the stated criterion.
\end{proof}

Recall that
\[
g_{a,c}=\gcd(ac+1,a+c),
\qquad
b_0=\frac{ac+1}{g_{a,c}}.
\]
The identity
\[
(a-1)(c-1)=(ac+1)-(a+c)
\]
shows that
\[
\gcd((a-1)(c-1),ac+1)=g_{a,c}.
\]
Consequently, \(b_0\) is the least positive
period of \(\phi_{a,c}\), and
\[
\phi_{a,c}(b_0)=0.
\]

It follows that stacking a sharp type with
an exact equality type preserves sharpness.
In particular, stacking with the standard
brick \(B_{a,c}\) always preserves sharpness.
If the minimal equality brick \(E_{a,c}\)
exists, it provides the shorter period
\(b_0\).

Thus sharp realizability for fixed outer
dimensions can be propagated from finitely
many residue representatives.

For \((a,c)=(4,5)\), we have
\[
g_{4,5}=3,\qquad b_0=7.
\]
The seven sharp residue representatives
constructed in
Corollary~\ref{cor:45-residues},
together with \(E_{4,5}=(4,7,5;3)\),
therefore realize every sharp type
\[
\left(
4,b,5;
\left\lceil\frac{3b}{7}\right\rceil
\right),
\qquad b\ge4.
\]

We now combine the available constructions
with the outer-pair envelope to obtain
intervals on which the numerical lower
bound \(L_3(n)\) is attained.

\begin{theorem}[Intervals of attainment]
	\label{thm:attainment-intervals}
	Over fields of characteristic zero,
	the minimized Laffey profile \(L_3(n)\)
	is attained for every
	\[
	14\le n\le52.
	\]
	Moreover, in characteristic zero, it is attained
	for every
	\[
	93\le n\le114.
	\]
\end{theorem}

\begin{proof}
	We treat the two intervals separately.
	
	\medskip
	\noindent
	\textbf{The interval \(14\le n\le52\).}
	
	The minimizing signatures below follow from
	the outer-pair envelope, with the additional
	ties at \(n=24,25\) verified directly in
	the numerical profile.
	
	For \(14\le n\le22\), a minimizing signature
	is
	\[
	(2,n-4,2).
	\]
	Write
	\[
	n-4=5q+r,
	\qquad 0\le r<5.
	\]
	The standard brick
	\[
	B_{2,2}=(2,5,2;4)
	\]
	gives, by stacking,
	\[
	B_{2,2}^{\star q}=(2,5q,2;4q).
	\]
	Applying the one-dimensional middle
	extension \(r\) times yields
	\[
	(2,5q+r,2;4q+r).
	\]
	This type is sharp, since
	\[
	\left\lceil
	\frac{4(5q+r)}5
	\right\rceil
	=
	4q+r
	\qquad(0\le r<5).
	\]
	Thus \(L_3(n)\) is attained throughout
	\(14\le n\le22\).
	
	At \(n=23\), a minimizing signature is
	\[
	(2,18,3).
	\]
	Besides
	\[
	B_{2,3}=(2,7,3;5),
	\]
	we use the sharp brick
	\[
	F=(2,4,3;3).
	\]
	Its explicit realization and maximality
	certificate are recorded in the
	Supplementary Material.
	The stack
	\[
	B_{2,3}^{\star2}\star F
	\]
	has type
	\[
	(2,18,3;13)
	\]
	and dimension
	\[
	1+2\cdot3+13
	=
	20=L_3(23).
	\]
	
	For \(24\le n\le52\), a minimizing
	signature may be chosen as
	\[
	(3,n-6,3).
	\]
	Indeed, at \(n=24,25\) this signature ties
	with the neighboring nearly balanced
	signature; for the remaining values it
	follows from the outer-pair envelope.
	
	Put
	\[
	b=n-6,
	\qquad 18\le b\le46.
	\]
	We use the sharp bricks
	\[
	\begin{aligned}
		C_4&=(3,4,3;3),\\
		C_5&=(3,5,3;3)=E_{3,3},\\
		C_8&=(3,8,3;5).
	\end{aligned}
	\]
	The brick \(C_5\) was constructed in
	Subsection~\ref{subsubsec:E33}; explicit
	realizations and maximality certificates
	for \(C_4\) and \(C_8\) are given in the
	Supplementary Material.
	
	Write
	\[
	b=5q+r,\qquad 0\le r<5.
	\]
	All five residue classes are realized by
	the following stacks:
	\[
	\begin{array}{c|l}
		r & \text{sharp realization}\\
		\hline
		0 & C_5^{\star q}\\
		1 & C_8^{\star2}\star C_5^{\star(q-3)}\\
		2 & C_4\star C_8\star C_5^{\star(q-2)}\\
		3 & C_8\star C_5^{\star(q-1)}\\
		4 & C_4\star C_5^{\star q}
	\end{array}
	\]
	For every occurring value
	\(18\le b\le46\), the exponents are
	nonnegative. Direct addition of the
	parameters gives middle dimension \(b\)
	and
	\[
	d=\left\lceil\frac{3b}{5}\right\rceil.
	\]
	Hence each construction attains Laffey's
	bound for the signature \((3,b,3)\).
	
	We conclude that \(L_3(n)\) is attained
	for every \(14\le n\le52\).
	
	\medskip
	\noindent
	\textbf{The interval \(93\le n\le114\).}
	
	By the outer-pair envelope, a minimizing
	outer pair throughout this interval is
	\[
	(a,c)=(4,5).
	\]
	Thus a minimizing signature is
	\[
	(4,n-9,5).
	\]
	The complete \((4,5)\)-sector theorem,
	Theorem~\ref{thm:45-complete-ray},
	provides a sharp realization over every
	field for every middle dimension
	\[
	b=n-9\ge4.
	\]
	Taking \(b=n-9\), we obtain
	\[
	L_3(n)
	=
	21+
	\left\lceil
	\frac{3(n-9)}7
	\right\rceil,
	\]
	and this value is realized by a local
	maximal commutative algebra with cube-zero
	radical.
	
	This proves attainment throughout
	\(93\le n\le114\), over every field.
\end{proof}

The two attainment intervals arise by different mechanisms. The first is obtained from a finite collection of sharp seed bricks, one-dimensional middle extensions, and stacking. The second follows directly from the complete \((4,5)\)-sector and the outer-pair envelope.

\subsection{Infinite attainment families}
\label{subsec:infinite-attainment}

The standard stacks attain the numerical Laffey
profile at infinitely many matrix dimensions.

\begin{theorem}[Infinite attainment families]
	\label{thm:infinite-attainment}
	For every integer \(k\ge2\), the value
	\(L_3(n)\) is attained at each of the five
	matrix dimensions listed in
	Table~\ref{tab:infinite-attainment}.
	These realizations exist over every field.
\end{theorem}

\begin{table}[ht]
	\centering
	\begin{tabular}{c|c|c}
		Realizing stack & \(n\) & \(L_3(n)\)\\
		\hline
		\(B_{k,k}^{\star k}\)
		& \(k^3+3k\)
		& \(3k^2+1\)\\
		\(B_{k,k}^{\star(k+1)}\)
		& \(k^3+k^2+3k+1\)
		& \(3k^2+2k+1\)\\
		\(B_{k,k}^{\star(k+2)}\)
		& \(k^3+2k^2+3k+2\)
		& \(3k^2+4k+1\)\\
		\(B_{k,k+1}^{\star(k+1)}\)
		& \(k^3+2k^2+4k+2\)
		& \(3k^2+4k+2\)\\
		\(B_{k,k+1}^{\star(k+2)}\)
		& \(k^3+3k^2+5k+3\)
		& \(3k^2+6k+3\)
	\end{tabular}
	\caption{Five infinite families for which
		the numerical Laffey profile is attained
		by standard stacks. The first row recovers
		the numerical parameters of Laffey's
		classical equality family.}
	\label{tab:infinite-attainment}
\end{table}

\begin{proof}
	By Theorems~\ref{thm:standard-brick}
	and~\ref{thm:stack}, the standard stack
	\(B_{a,c}^{\star q}\) has
	\[
	n=a+c+q(ac+1),
	\qquad
	\dim A=1+ac+q(a+c).
	\]
	These formulas give the dimensions
	displayed in the table.
	
	It remains to show that they coincide
	with the numerical minimum \(L_3(n)\).
	
	For the first two rows, the outer pair
	\((k,k)\) lies on the continuous upper
	envelope. The fourth row is the upper
	endpoint of the envelope interval for
	\((k,k+1)\). Hence these three stacks
	attain \(L_3(n)\) directly.
	
	For the third row, put
	\[
	n=k^3+2k^2+3k+2.
	\]
	The continuous envelope is attained
	by \((k,k+1)\), but
	\[
	\gamma_{k,k+1}(n)-\gamma_{k,k}(n)
	=
	\frac{k^2-1}{k^2+k+1}
	\in(0,1).
	\]
	Since \(\gamma_{k,k}(n)\) is an integer,
	\[
	\lfloor\gamma_{k,k+1}(n)\rfloor
	=
	\lfloor\gamma_{k,k}(n)\rfloor.
	\]
	Thus the balanced stack also attains
	\(L_3(n)\).
	
	Similarly, for the fifth row,
	\[
	n=k^3+3k^2+5k+3,
	\]
	the continuous envelope is attained
	by \((k+1,k+1)\), whereas
	\[
	\gamma_{k+1,k+1}(n)
	-
	\gamma_{k,k+1}(n)
	=
	\frac{k(k+2)}{k^2+2k+2}
	\in(0,1).
	\]
	Since \(\gamma_{k,k+1}(n)\) is integral,
	the two pairs yield the same value
	after taking floors.
	
	Consequently, every stack in the table
	attains the numerical profile \(L_3(n)\).
	All standard bricks and their stacks
	are defined over every field, so the
	assertion holds in arbitrary characteristic.
\end{proof}

The last two comparisons illustrate the effect of rounding in Laffey's bound: a stack may attain the discrete minimum even when its outer pair does not maximize the continuous envelope. Thus attainment extends beyond the envelope's exact intersection points and occurs at infinitely many matrix dimensions.

\begin{figure}[htbp]
	\centering
	\includegraphics[width=\textwidth]{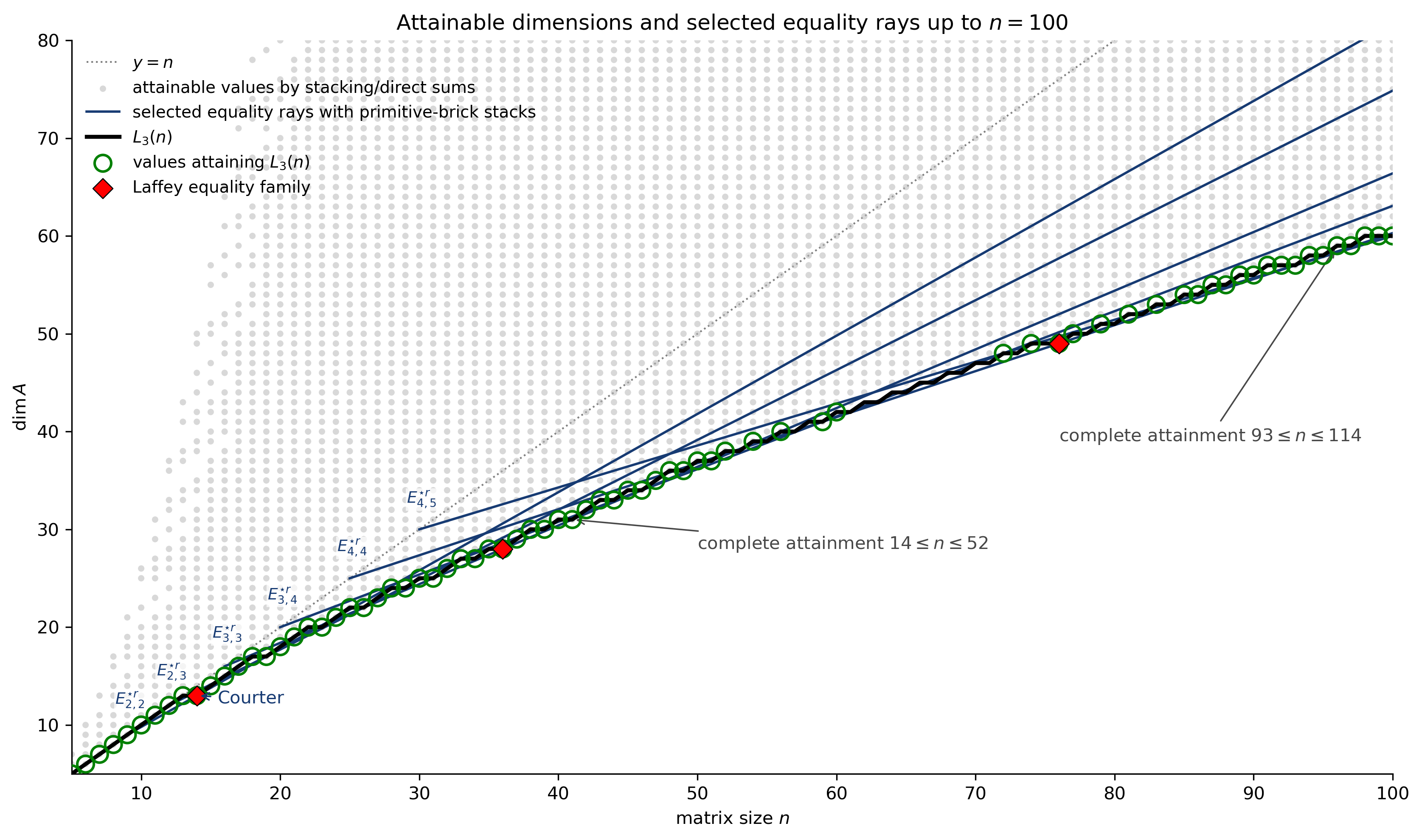}
	\caption{Attainability landscape in the
		$(n,\dim A)$-plane for $n\le100$.
		Gray points represent dimensions attained
		by the constructions considered in this paper,
		including finite direct sums.
		The black curve shows the numerical Laffey
		profile $L_3(n)$ for local maximal commutative
		algebras with cube-zero radical.
		Navy lines indicate selected equality rays
		generated by stacks of minimal equality bricks.
		Green open circles mark attained values of
		$L_3(n)$, while red diamonds identify
		Laffey's classical equality family
		$n=k^3+3k$, $\dim A=3k^2+1$.}
	\label{fig:attainability}
\end{figure}

Figure~\ref{fig:attainability} summarizes the
attainability results and places the constructed
equality families in the broader landscape of
commutative matrix algebras.

\section{Relation to classical constructions}

Stacking is related to earlier constructions of
maximal commutative matrix algebras, but its
defining feature is the identification of the
outer Loewy layers and the full block
$\Sigma=\operatorname{Hom}(U,Z)$.
In Courter's exponent-three unfolding
\cite{Courter1973}, the first radical layer is
replicated while the square of the radical is
retained. These operations need not coincide:
Courter's construction preserves $J^2$, whereas
stacking shares all of $\Sigma$, and the inclusion
$J^2\subseteq\Sigma$ may be strict.

The standard bricks and their iterated stacks
are also related to earlier module-theoretic
constructions. As shown in Remark~2, they
can be recovered from construction $C_1$ of
Nowak-Morales \cite{NowakMorales1999}.
They also reproduce the numerical parameters
of Laffey's classical equality family
\cite{Laffey1985}. These correspondences concern
the constructions or their parameter types;
equality of parameters alone does not imply
isomorphism of the resulting algebras.

The contribution of the present approach is
the reduced-centralizer formulation of stacking
for arbitrary compatible cube-zero algebras,
together with the construction of primitive
equality bricks and the arithmetic analysis of
sharpness. This framework extends the classical
constructions from particular families to
systematic questions of realizability and
attainment of Laffey's bound.

\section{Conclusion and further directions}
\label{sec:conclusion}

We have developed an elementary framework for constructing local maximal commutative matrix algebras with cube-zero radical and for studying equality in Laffey's signature bound.

For every pair \(a,c\ge2\), the standard brick
\[
B_{a,c}=(a,ac+1,c;a+c)
\]
attains the bound, and stacking produces further exact equality types. On the balanced branch, all minimal equality bricks exist in characteristic zero. For even \(k\), 
\(E_{k,k}=B_{k,k}\), while for odd \(k\),
\[
E_{k,k}
=
\left(k,\frac{k^2+1}{2},k;k\right),
\qquad
E_{k,k}^{\star2}\sim_{\mathrm{type}}B_{k,k}.
\]

The outer-pair envelope shows that balanced and nearly balanced outer layers suffice to compute the numerical profile \(L_3(n)\).  The first nontrivial nearly balanced brick, $E_{4,5}=(4,7,5;3)$, exists over every
field. In characteristic zero, a complete
set of seven sharp residue representatives
yields sharp realizations throughout the
$(4,5)$-sector for $b\ge4$.

Consequently, $L_3(n)$ is attained for
$14\le n\le52$ and $93\le n\le114$
in characteristic zero, and along five
infinite families over every field.

Consequently, \(L_3(n)\) is attained for \(14\le n\le52\) in characteristic zero, for \(93\le n\le114\) over every field, and along five infinite families of matrix dimensions over every field.

The next unresolved primitive-root cases on the nearly balanced branch are
\[
E_{k,k+1},\qquad k\equiv1\pmod3,\quad k\ge7.
\]
More generally, it remains open whether \(L_3(n)\) is attained for all sufficiently
large \(n\), or even for every \(n\ge14\).

The finite exceptional constructions are supported by exact reduced-centralizer certificates. Thus computation is used to verify explicit algebraic constructions,
rather than as a substitute for proof.

\section*{Data and code availability}

The exact representatives, verification scripts, and machine-readable data used in the computations 
and figures are available at \url{https://doi.org/10.6084/m9.figshare.33301656}.

\section*{Declaration of generative AI and AI-assisted technologies in manuscript preparation}

During the preparation of this work, the author used ChatGPT (OpenAI) to assist with language editing, manuscript organization, and the drafting and refinement of computational verification and plotting code. The author critically reviewed, edited, and independently verified the mathematical arguments, computational results, and final content, and takes full responsibility for the manuscript.

\end{document}